\documentclass[11pt]{amsart}
\usepackage{amsmath}
\usepackage{amsfonts}
\usepackage{amssymb}
\usepackage{verbatim}
\usepackage{enumerate}
\usepackage{graphicx}
\usepackage{longtable}
\usepackage[all,cmtip]{xy}
\usepackage{upref}
\usepackage{fixltx2e}
\usepackage{hyperref}
\usepackage[hang,small,bf]{caption}

\newtheorem{definition}{Definition}[section]
\newtheorem{theorem}[definition]{Theorem}
\newtheorem{proposition}[definition]{Proposition}

\newtheorem{lemma}[definition]{Lemma}

\renewcommand{\theta}{\vartheta}

\newcommand{\T}{{\mathbb T}}

\newcommand{\R}{{\mathbb R}}
\newcommand{\N}{{\mathbb N}}

\renewcommand{\phi}{\varphi}

\begin{document}

\title[Periodic orbits of non-exact oscillating magnetic fields on surfaces]{Infinitely many periodic orbits in non-exact oscillating magnetic fields on surfaces with genus at least two for almost every low energy level}

%    author information
\author{Luca Asselle}
\address{Ruhr-Universit\"at Bochum, Fakult\"at f\"ur Mathematik, NA 4/35, D-44780 Bochum, Germany}
\email{\href{mailto:luca.asselle@ruhr-uni-bochum.de}{luca.asselle@ruhr-uni-bochum.de}}
%\thanks{}
\author{Gabriele Benedetti}
\address{WWU M\"unster, Mathematisches Institut, Einsteinstrasse 62, D-48149 M\"unster, Germany}
%\curraddr{}
\email{\href{mailto:benedett@uni-muenster.de}{benedett@uni-muenster.de}}


\date{\today}
%\dedicatory{}
\begin{abstract}
In this paper we consider oscillating non-exact magnetic fields on surfaces with genus at least two and show that for almost every energy level $k$ below a certain value $\tau_+^*(g,\sigma)$ less than or equal to the \textit{Ma\~n\'e critical value of the universal cover} there are infinitely many closed magnetic geodesics with energy $k$.
\end{abstract}

\maketitle

\section{Introduction}

Let $(M,g)$ denote a closed connected orientable Riemannian surface with genus at least two. Let $\sigma\in \Omega^2(M)$ be a non-exact 2-form, let $\omega_g$ denote the standard symplectic form on $TM$ obtained by pulling back the canonical symplectic form $dp\wedge dq$ on $T^*M$ via the Riemannian metric and let 
\[\omega:=\ \omega_g + \pi^*\sigma \]
be the \textit{twisted symplectic form} determined by the pair $(g,\sigma)$. Let $E:TM\rightarrow \R$ denote the energy Hamiltonian 
\[E(q,v)\ =\ \frac{1}{2}\, \|v\|_q^2\]
and $\phi_t:TM\rightarrow TM$ represent the flow of the symplectic gradient of $E$ with respect to $\omega$. We call $\phi_t$ the \textit{magnetic flow} of the pair $(g,\sigma)$. The reason for this terminology is that $\phi_t$ can be thought of as modeling the motion of a particle of unit mass and charge under the effect of a magnetic field represented by the 2-form $\sigma$. 

The aim of this paper is to generalize the main theorem in \cite{AMMP14} to the non-exact oscillating case, that is to prove the existence of infinitely many closed periodic orbits on almost every low energy level 
without any additional non-degeneracy assumption (such an assumption is needed in \cite{AMP13} for the exact case).
More precisely, to any magnetic pair $(g,\sigma)$ we can associate the energy value $\tau_+^*(g,\sigma)$, which is strictly positive if and only if the 2-form $\sigma$ changes sign (see Lemma \ref{lem:tau}). 
A $2$-form $\sigma$ of this kind is called \textit{oscillating}. Then the following theorem holds.

\begin{theorem}
Let $(M,g)$ be a closed orientable Riemannian surface of genus at least two and let $\sigma\in \Omega^2(M)$ be a non-exact oscillating $2$-form. For almost every $k\in (0,\tau_+^*(g,\sigma))$ 
there are infinitely many geometrically distinct closed magnetic geodesics with energy $k$.
\label{mainteo}
\end{theorem}

\vspace{3mm}

In \cite{AMMP14} the same result has been shown when the 2-form $\sigma$ is exact. There the energy interval $(0,\tau_+^*(g,\sigma))$ is replaced by the interval $(0,c_u(L))$, where $c_u(L)$ denotes the Ma\~{n}\'e critical value of the universal cover of the Lagrange function $L:TM\rightarrow \R$ given by $L(q,v) = \|v\|_q^2/2 + \lambda_q(v),$ where $\lambda\in\Omega^1(M)$ is a primitive of $\sigma$. The Lagrangian $L$ gives rise to an associated free-period action functional 
\[S_k:\Lambda \longrightarrow \R\, , \ \ \ \ (x,T)\ \longmapsto\ T\, \int_0^1 \Big [ L\Big (x(t),\frac{\dot x(t)}{T}\Big ) + k \Big ]\, dt\, ,\]

\noindent on the space $\Lambda:= H^1(\T,M)\times (0,+\infty)$. The interest in this functional relies on the fact that its critical points are exactly the closed magnetic geodesics with energy $k$. The properties of $S_k$ are very different according to whether $k$ is supercritical or subcritical (for a beautiful survey, see \cite{Abb13}).

For $k>c_u(L)$, there exists a global minimizer in every non-trivial free homotopy class since $S_k$ is bounded from below on a fixed free homotopy class and it satisfies the Palais-Smale condition \cite{CIPP00}.

For $k<c_u(L)$, the functional is not bounded from below on each free homotopy class and there are examples where the Palais-Smale property fails. However, following Contreras' general result on autonomous Tonelli Lagrangian systems \cite{Con06}, in the subcritical range $S_k$ has a mountain-pass geometry on the space of contractible loops. The two valleys are represented by the constant loops and by the loops with negative action. Moreover, for the specific case of exact magnetic flows on surfaces the results of Taimanov \cite{Tai92b,Tai92a,Tai93} and independently of Contreras, Macarini and G. Paternain \cite{CMP04} imply that, for $k$ in this range, there exists a local minimizer $\alpha_k$ for $S_k$. If we suppose further that the minimizer is strict, then $S_k$ has a mountain-pass geometry for every iterate $\alpha^n_k$ of $\alpha_k$, since $\alpha^n_k$ is still a strict local minimizer. For every $n\geq1$ the two valleys are a small neighborhood of $\alpha^n_k$ and the set of loops in the same free 
component of $\alpha^n_k$ with action less than $S_k(\alpha^n_k)$. 
Therefore, critical points of mountain-pass type do exist provided one can prove the convergence of the Palais-Smale sequences associated to the minimax values. An argument originally due to Struwe \cite{Str90} (and used already for the free-period action functional by Contreras \cite{Con06}) ensures such convergence on almost every energy level below $c_u(L)$ exploiting the fact that the functional $S_k$ is monotone in $k$ (see \cite{AMMP14}). Combining this result with a study of the behaviour of the index of a critical point under iteration, the analogue of Theorem \ref{mainteo} for exact magnetic flows on subcritical energy levels has been proved in \cite{AMMP14}. 
\medskip

In the present paper we aim to prove Theorem \ref{mainteo} exactly by generalizing the scheme of proof carried out in \cite{AMMP14} to our setting. The first difficulty is that in the non-exact case the Lagrangian action functional is not available anymore, because the closed 2-form $\sigma$ does not admit a primitive. However, we can define a closed $1$-form $dS_k$ on $\Lambda$, whose vanishing points still correspond to magnetic geodesics. In general, the direct study of the vanishing set of the $1$-form would require to apply the methods of Novikov theory for multivalued functionals to this infinite-dimensional setting, but no result has been obtained so far by pursuing this strategy.

The main idea to overcome this problem in our case is due to Will Merry \cite{Mer10}. He shows that, if the integral of $\sigma$ vanishes on every $2$-torus $\mathbb T^2\rightarrow M$, the closed $1$-form $dS_k$ is actually exact on $\Lambda$ and a primitive $S_k:\Lambda\rightarrow \R$ can be explicitly written. Merry proved that the integrals of $\sigma$ on $2$-tori vanish provided the \textit{Ma\~{n}\'e critical value of the universal cover} $c(g,\sigma)\in \R\cup \{+\infty\}$ associated to the magnetic pair $(g,\sigma)$ is finite \cite[Lemma 2.2]{Mer10}. Since on surfaces of genus at least two the Ma\~{n}\'e critical value is always finite (see Lemma \ref{manfin}), we have the existence of the functional $S_k$ in our case. We recall its basic properties in relation to the Palais-Smale condition in Section \ref{sec:pal}.

Moreover, $S_k$ has locally (in particular near a critical point) the same structure as a Lagrangian action functional (with a primitive $\theta$ not defined on the whole $M$). Therefore, the local theory is the same as in the exact case: iterates of (strict) local minimizers are still (strict) local minimizers (Lemma \ref{persistenceoflocalminimizers}) and the Morse index of the critical points satisfies the same iteration properties as described in \cite[Section 1]{AMP13} and in \cite{AMMP14}. In particular, as shown in \cite{AMMP14} for the exact case, a sufficiently high iterate of a periodic orbit cannot be a mountain pass critical point of $S_k$ (see Section \ref{sec:ite} for further details).

We restrict our study to energy values below $\tau^*_+(g,\sigma)$ since, in this range, $S_k$ is not bounded from below and there exists a local minimizer $\alpha_k$ of $S_k$ by \cite{Tai93} or \cite{CMP04} (see Section \ref{sec:min}). In Section \ref{sec:minmax} we will use the existence of this local minimizer and its iterates to construct a sequence of minimax values, which allows us to prove Theorem \ref{mainteo}.

We shall mention that in \cite{GGM14} it has been proved that, for surfaces of genus at least two with a nowhere-vanishing magnetic field, the twisted geodesic flow has infinitely many periodic orbits on every low energy level. There the fact that for nowhere-vanishing magnetic fields the energy levels are of contact type for low energies was essentially used. Therefore, this kind of result cannot be applied here, since as we will see in Proposition \ref{nocontact}, low energy levels are never of contact type for oscillating $2$-forms.

Finally, we briefly point out where in our argument it is crucially needed that the base manifold $M$ has dimension two. 

First, the existence of a local minimizer for $S_k$ is so far only known when $M$ is a surface and relies essentially on the fact that in dimension two, one can define an analogue of the functional $S_k$ called $\mathcal T_k$ on the space of oriented multi-curves, that are the boundary of surfaces embedded in $M$ (see Section \ref{sec:min}). Moreover, the critical points of the new functional are still multi-curves whose components are magnetic geodesics with energy $k$.
Although $S_k$ is not bounded from below on the connected components of the free loop-space, the functional $\mathcal T_k$ is and one might try to apply the direct method of the calculus of variation to find its critical points. The space of embedded multi-curves is however not the best to work with and that is the reason why Taimanov in \cite{Tai92b,Tai92a,Tai93} considered the restriction of $\mathcal T_k$ to the space of the so-called \textit{films}, whilst Contreras, Macarini and Paternain in \cite{CMP04} took a suitable extension of $\mathcal T_k$ to the space of integral 2-currents.

Second, the fact that (strict) local minimizers for $S_k$ remain (strict) local minimizers also when iterated is in general true only on orientable surfaces (however, one can prove Theorem \ref{mainteo} also for non-orientable surfaces by working on the orientation double covering). 
The interested reader may look at \cite{Hed32} for a counterexample in dimension greater than two and at \cite[Example 9.7.1]{KH95} for a counterexample in the non-orientable case. 

These are two of the reasons why in dimension bigger than two much less is known about the existence of closed magnetic geodesics. One of the most general results in higher dimension is the existence of one contractible periodic orbit on \textit{almost} every low energy level if $\sigma\neq0$ \cite{Sch06}. Existence for \textit{every} low energy can be obtained by requiring $\sigma$ to be symplectic \cite{GG09,Ush09}. Finally, multiplicity results hold if $\sigma$ is a K\"ahler form \cite{Ker99}. We also refer to the beautiful survey \cite{Abb13}, or to \cite{Con06}, for a discussion of the dynamics of exact magnetic forms in any dimension within the abstract framework of Tonelli Lagrangians.

\subsection*{Acknowledgements}
We are very grateful to our PhD advisors, A. Abbondandolo and G. P.\ Paternain respectively, for many fruitful discussions. We warmly thank the referee for carefully reading the draft and for suggesting various modifications which allowed us to improve the paper.

%%%%%%%%%%%%%%%%%%%%%%%%%%%%%%%%%%%%%%%%%%%%%%%%%%%%%%
%%%%%%%%%%%%%%%%%%%%%%%%%%%%%%%%%%%%%%%%%%%%%%%%%%%%%%
%%%%%%%%%%%%%%%%%%%%%%%%%%%%%%%%%%%%%%%%%%%%%%%%%%%%%%

\section{Preliminaries}\label{sec:pre}

The contents of Sections 2--5 hold more generally for any closed connected Riemannian manifold $(M,g)$. Set  
\[\Lambda := \ H^1(\T,M)\times (0,+\infty).\]

We denote by $\widetilde {\Lambda}$ the same object on the universal cover $\widetilde M$ and by $p:\widetilde {\Lambda} \rightarrow \Lambda$ the projection map.
The space $\Lambda$ represents the set of loops with arbitrary period through the identification

\[(x,T)\in\Lambda\ \ \ \Longleftrightarrow\ \ \ \gamma:[0,T]\rightarrow M\, , \ \ \gamma(t):=x\big(t/T\big)\,.\]

For every free homotopy class of loops $\nu\in[\T,M]$, let $\Lambda_\nu$ be the connected component of $\Lambda$ given by loops in the class $\nu$ and denote by $\Lambda_0$ the connected component of contractible loops. We endow $\Lambda$ with the structure of a Hilbert manifold given by the product metric. The space $\Lambda$ carries a natural $\N$-action given by iteration
\begin{equation}
\psi:\N\times \Lambda \longrightarrow \Lambda\, , \ \ \ \ (n,(x,T))\longmapsto \psi(n,(x,T)) = \psi^n (x,T) := (x^n,nT)\,,
\label{iterationmap}
\end{equation}

\noindent where $x^n(t):=x(nt)$. In particular, for every class $\nu$, there exists a unique class $\nu^n$ such that $\psi^n(\Lambda_{\nu})\subset \Lambda_{\nu^n}$ and 
$\psi^n:\Lambda_\nu\rightarrow\Lambda_{\nu^n}$ is an embedding of Hilbert manifolds.

Let $\sigma\in \Omega^2(M)$ be a closed 2-form and let $\tilde \sigma$ be its pull-back to $\widetilde M$. We assume $\sigma$ to be \textit{weakly exact}, in the sense that $\tilde \sigma$ is exact (this is equivalent to requiring that $\sigma |_{\pi_2(M)}=0$). 
Denote by $\phi_t:TM\rightarrow TM$ the magnetic flow of the pair $(g,\sigma)$, i.e.\ the flow of the symplectic gradient of the Hamiltonian given by the kinetic energy with respect to the twisted symplectic form 
\[\omega := \ \omega_g + \pi^* \sigma\, ,\]

\noindent where $\omega_g$ is the pull-back on $TM$ of the standard symplectic form on $T^*M$ via the Riemannian metric and $\pi:TM\rightarrow M$ is the tangent bundle. We are interested in finding periodic orbits for $\phi_t$. 
We call $\gamma$ a \textit{closed magnetic geodesic} of the pair $(g,\sigma)$ if it is the projection under $\pi$ of a periodic orbit of $\phi_t$.

Now, fix a primitive $\theta$ of $\tilde \sigma$ and consider the Lagrangian 
\[L:T\widetilde M \longrightarrow \R\, ,\ \ \ \ L(q,v) := \ \frac{1}{2}\, \|v\|_q^2 + \theta_q (v)\, .\]

The Euler-Lagrange flow of $L$ is precisely the lifted flow $\tilde{\phi}_t:T\widetilde M \rightarrow T\widetilde M$ of the magnetic flow $\phi_t:TM\rightarrow TM$ (see for instance \cite{CI99}). 
Consider the action functional $\widetilde S_k:\widetilde{\Lambda} \rightarrow\R$ defined by
\begin{equation*}
\widetilde S_k(\tilde x,T)\ :=\ T\, \int_{\T} \Big [L\Big (\tilde x(t),\frac{\dot {\tilde x}(t)}{T}\Big ) + k \Big ]\, dt
\label{ak}
\end{equation*}

\noindent and define the \textit{Ma\~{n}\'e  critical value} of the pair $(g,\sigma)$ as 
\begin{equation*}
c(g,\sigma):= \ \sup \Big \{  k\in \R \ \Big |\ \exists \ \tilde \gamma \in \widetilde{\Lambda} \ \ \text{s.t.} \ \ \widetilde S_k (\tilde \gamma) < 0\, \Big \}\ \in \ \R \cup \{+\infty\}\, ,
\end{equation*}

\noindent where as usual we make the identification $\tilde \gamma = (\tilde x,T)$. The next lemma implies that $c(g,\sigma)$ is finite if and only if the lift $\tilde \sigma$ of $\sigma$ to the universal cover 
admits a bounded primitive (see \cite{BP02} for a proof).

\begin{lemma}\label{lemma2.1}
We can alternatively define $c(g,\sigma)$ as follows 
\begin{equation}
c(g,\sigma) \ =\ \inf_{u\in C^\infty(\widetilde M)} \sup_{q\in \widetilde M} \frac{1}{2}\, \| d_qu - \theta_q\|^2\, .
\label{alternativedefinition}
\end{equation}
\end{lemma}
Following \cite{Mer10} we now introduce a second functional $S_k:\Lambda\rightarrow\R$ when $c(g,\sigma)<\infty$. The key observation is the following

\begin{lemma}
If $c(g,\sigma)<\infty$, then for any $f:\T^2\rightarrow M$ smooth, $f^*\sigma$ is exact.
\label{supercazzola}
\end{lemma}
We refer to \cite{Mer10} and references therein, in particular \cite{Pat06}, for a proof and a discussion of this result.

For every $\nu\in [\T,M]$ pick a reference loop $x_\nu$ in the connected component of $H^1(\T,M)$ corresponding to $\nu$. Given any other loop $x\in H^1(\T,M)$ in the same connected component, let $C(x)$ denote a cylinder connecting $x_\nu$ to $x$ and define $S_k:\Lambda_\nu \longrightarrow \R$ by 
\begin{equation}
S_k(x,T):= \ T\, \int_0^1\Big (\frac{1}{2T^2}\, |\dot x(t)|^2+k\Big )\, dt + \int_{C(x)}\sigma\,.
\label{primadefsn}
\end{equation}

\noindent This is well defined because the value of the integral of $\sigma$ over $C(x)$ is independent of the choice of the cylinder. Indeed, if $C'(x)$ is another cylinder with the same boundary, then 
\[\T^2(x) := \ C(x) \cup \overline{C'(x)}\]

\noindent is a torus (where $\overline{C'(x)}$ denotes the cylinder $C'(x)$ taken with opposite orientation) and hence the integral of $\sigma$ over $\T^2(x)$ is zero by Lemma \ref{supercazzola}. Observe that $S_k$ depends on the choice of the reference loop. However, if we change $x_\nu$ to $x'_\nu$, the functional only changes by the addition of a constant (the integral of $\sigma$ over a cylinder connecting $x_\nu$ and $x'_\nu$) and, hence, the geometric properties of $S_k$ are unchanged.

When $\nu=0$ is the class of contractible loops, we make the natural choice of taking as reference loop $x_0$ a fixed constant loop. Then, if $\gamma=(x,T)\in\Lambda_0$, the cylinder $C(x)$ is just a capping disc for $x$. As a consequence, if $\gamma'\in\Lambda$ is such that $\gamma'(0)=\gamma(0)$, then
\begin{equation}\label{eq:conc}
S_k(\gamma\ast\gamma')=S_k(\gamma)+S_k(\gamma')\,,
\end{equation}
where $\ast$ denotes the concatenation of loops.

Moreover, the functional $S_k\big|_{\Lambda_0}$ is related with $\widetilde S_k$ by the formula 
\[S_k(p(\tilde x,T))\ =\ \widetilde S_k(\tilde x,T)\, ,\]

\noindent so that $S_k|_{\Lambda_0}$ is unbounded from below if $k<c(g,\sigma)$. By Equation \eqref{eq:conc}, we get then the following result.

\begin{lemma}\label{lem:unb}
If $k<c(g,\sigma)$, then $S_k|_{\Lambda_\nu}$ is unbounded from below for every $\nu\in[\T,M]$.   
\end{lemma}

It is a straightforward computation to show that the functional $S_k$ is of class $C^2$ (actually smooth). We refer to \cite{Mer10} for the details.

The following lemma, contained in \cite[Corollary 2.3]{Mer10}, explains why we are interested in the functional $S_k$.

\begin{lemma}
A loop $\gamma\in\Lambda$ is a critical point for $S_k$ if and only if it is a magnetic geodesic with energy $k$. 
\end{lemma}

The main result known so far about the critical points of $S_k$ in this generality is Theorem 1.1 of \cite{Mer10}, which we now state.

\vspace{2mm}

\begin{theorem}[W. Merry, 2010]\label{merteo}
Let $(M,g)$ be a closed connected Riemannian manifold and let $\sigma\in \Omega^2(M)$ be a closed weakly exact 2-form. Let $c:=c(g,\sigma)\in \R\cup \{+\infty\}$ denote the Ma\~{n}\'e critical value and $\phi_t$ the magnetic flow defined by $\sigma$. Suppose in addition that $c<\infty$. Then:

\begin{enumerate}
\item if $k>c$, for each non trivial free-homotopy class $\nu\in [\T,M]$, there is a closed orbit of $\phi_t$ with energy $k$ such that its projection to $M$ belongs to $\nu$;

\item for a.e.\ $k\in (0,c)$, there is a closed orbit of $\phi_t$ with energy $k$ such that its projection to $M$ is contractible.
\end{enumerate}

\end{theorem}

We finish this section by describing explicitly the behaviour of $S_k$ under iteration.

\begin{lemma}\label{lem:equi}
For every $\nu\in[\T,M]$ and $n\in\N$, there exists $b(\nu,n)\in\R$, depending on the reference loops $\{x_{\nu^n}\}_{n\in\N}$, such that
\begin{equation}\label{equivariance}
S_k(\psi^n(\gamma))=n\cdot S_k(\gamma)+b(\nu,n)\,,\quad\forall\, \gamma\in\Lambda_\nu\,.
\end{equation}
Choosing $x_{\nu^n}=x^n_\nu$ when $\nu$ has infinite order, or $x_\nu$ equal to a constant path when $\nu=0$, yields $b(\nu,n)=0$ for every $n\in\N$. In this case, $S_k$ is $\N$-equivariant:
\begin{equation}\label{almequivariance}
S_k(\psi^n(\gamma))=n\cdot S_k(\gamma)\,,\quad\forall\, \gamma\in\Lambda_\nu\,.
\end{equation}
\end{lemma}

\begin{proof}
Let $C(x)$ be a cylinder connecting $x_\nu$ to $x$ and denote by $C^n(x)$ the cylinder obtained by iterating each loop $n$ times. We construct a cylinder $C(x^n)$ from $x_{\nu^n}$ to $x^n$ as the concatenation of a cylinder $C(x_\nu^n)$ connecting $x_{\nu^n}$ with $x_\nu^n$ and the cylinder $C^n(x)$ connecting $x_\nu^n$ with $x^n$. We compute
\[\int_{C(x^n)}\sigma\ =\ \int_{C(x_\nu^n)\cup C^n(x)}\sigma\ =\ \int_{C(x_\nu^n)}\sigma\ +\ \int_{C^n(x)}\sigma\ =\ b(\nu,n)+n\int_{C(x)}\sigma\,,\]

\noindent where 
\[b(\nu,n):=\ \int_{C(x_\nu^n)}\sigma\, .\]

Plugging this identity into the definition of $S_k(\gamma^n)$, we get
\begin{align*}
S_k(\gamma^n)&=\ \int_0^{nT}\Big (\frac{1}{2}\, |\dot \gamma(t')|^2+k\Big )\, dt'\ +\ \int_{C(x^n)}\sigma\\
&=\ n\int_0^{T}\Big (\frac{1}{2}\, |\dot \gamma(t')|^2+k\Big )\, dt'\ +\ b(\nu,n)\ +\ n\int_{C(x)}\sigma\\
&=\ n\cdot S_k(\gamma)\ +\ b(\nu,n)\,.\qedhere
\end{align*}
\end{proof}

%%%%%%%%%%%%%%%%%%%%%%%%%%%%%%%%%%%%%%%%%%%%%%%%%%%%%
%%%%%%%%%%%%%%%%%%%%%%%%%%%%%%%%%%%%%%%%%%%%%%%%%%%%%
%%%%%%%%%%%%%%%%%%%%%%%%%%%%%%%%%%%%%%%%%%%%%%%%%%%%%

\section{The Palais-Smale condition for \texorpdfstring{$S_k$}{S}}\label{sec:pal}

In this section we recall the basic properties of the functional $S_k$. All the results we mention are contained and proved in \cite{Mer10} and in \cite{AMP13}. They hold for any closed Riemannian manifold $M$ with $c(g,\sigma)<\infty$.

\begin{definition}
Let $(\mathcal H,\langle \cdot,\cdot \rangle)$ be a Riemannian Hilbert manifold and consider $S:\mathcal H\rightarrow \R$ a functional of class $C^1$. We say that $S$ satisfies the $\mathsf{Palais}$-$\mathsf{Smale\ condition}$, if every sequence $\{x_n\}\subseteq \mathcal H$ satisfying 
\[\|dS(x_n)\| \ \longrightarrow \ 0\, ,\ \ \ \ \sup_{n\in \N} \ \big |S(x_n)\big |\ <\ \infty\]
admits a convergent subsequence. We say that $S$ satisfies the Palais-Smale condition $\mathsf{at\ level}\ \mu\in \R$, if every sequence $\{x_n\}\subseteq \mathcal H$ satisfying 
\[\|dS(x_n)\| \ \longrightarrow \ 0\, ,\ \ \ \ S(x_n)\ \longrightarrow \ \mu\]
admits a convergence subsequence.
\end{definition}

\vspace{2mm}

The first result we present states that the functional $S_k$ satisfies locally the Palais-Smale condition. Palais-Smale sequences $(x_n,T_n)$ with $\liminf T_n =0$ are a possible source of non-compactness, but it turns out that they only occur at level zero.

\begin{theorem}\label{teo:pal}
Suppose $(x_n,T_n)\subseteq \Lambda$ is a Palais-Smale sequence with $T_n$ bounded from above. Then the following hold:

\begin{enumerate}
\item If $\liminf T_n >0$, then, passing to a subsequence if necessary, the sequence $(x_n,T_n)$ is convergent in the $H^1$-topology.

\item If $\liminf T_n =0$, then, passing to a subsequence if necessary, 

$$S_k (x_n,T_n)\ \longrightarrow \ 0\, .$$
\end{enumerate}
\label{Palaissmale}
\end{theorem}

We refer to \cite{Mer10} for the proof. In the light of this result, we see that $S_k$ satisfies the Palais-Smale condition provided Palais-Smale sequences have bounded period. A possible generalization of Proposition F in \cite{Con06} would yield that if the energy level $E^{-1}(k)$ is of contact type, then the Palais-Smale condition holds. Unfortunately, for sufficiently low energies, $E^{-1}(k)$ is never of contact type when the form is oscillating, as we will see in Section \ref{cha:con}.
Hence, to overcome the possible lack of the Palais-Smale condition, in Chapter \ref{cha:fin} we will use an argument due to Struwe \cite{Str90} to construct Palais-Smale sequences with bounded period on almost all energy levels below $\tau^*_+(g,\sigma)$.
\medskip

Recall that $\alpha =(x,T)\in \Lambda$ is said to be a \textit{strict local minimizer} of $S_k$ on $\Lambda$, if the $\T$-orbit of $\alpha$ 
\[\T \cdot \alpha := \ \Big \{ \big (x(\tau + \cdot), T\big )\ \Big |\ \tau \in \T\Big \}\]
has a neighborhood $\mathcal U$ in $\Lambda$, such that 
\[S_k(\gamma) \ >\ S_k(\alpha)\, ,\ \ \ \ \forall \ \gamma \in \mathcal U\setminus \T\cdot \alpha\, .\]

The fact that the Palais-Smale condition holds locally allows to prove that the $\T$-orbit of a strict local minimizer $\alpha$ has neighborhoods on whose boundary the infimum of $S_k$ is strictly larger than $S_k(\alpha)$. 
The proof given in \cite[Lemma 4.3]{AMP13} goes through without any change.

\begin{lemma}\label{lem:strict}
Let $\alpha=(x,T)$ be a strict local minimizer for $S_k$ on $\Lambda$. If the neighborhood $\mathcal U$ of $\T\cdot \alpha$ is sufficiently small then 
\[\inf_{\partial \mathcal U} \ S_k \ >\ S_k(\alpha)\, .\]
\label{neighborhood}
\end{lemma}

\vspace{-5mm}

The next result ensures that the flow of $-\nabla S_k$ is relatively complete on $\Lambda'$, the complement in $\Lambda$ of the space of contractible loops $\Lambda_0$. Since $\Lambda$ is not complete we may not expect the flow of $-\nabla S_k$ to be positively complete. However, the second statement of the next lemma ensures that it is relatively complete in $\Lambda_0$ on any interval that does not contain zero.

\begin{lemma}
Let $k>0$ and let $\Lambda= \Lambda'\sqcup \Lambda_0$. The following two properties hold:

\begin{enumerate}
\item The sublevels $\{S_k\leq c\}$ are complete on $\Lambda'$. More precisely, if $(x_n,T_n)\subseteq \Lambda'$ is such that $T_n\rightarrow 0$, then $S_k(x_n,T_n)\rightarrow +\infty$.

\item If $[a,b]\subseteq \R$ is an interval such that $0\not \in [a,b]$, then the local flow of $-\nabla S_k$ is relatively complete on $\Lambda_0\cap \{a\leq S_k\leq b\}$.
\end{enumerate}
\label{positivecomplete}
\end{lemma}

\vspace{-3mm}

The proof of the first statement follows directly from the proof of Lemma 3.2 in \cite{Abb13} (there the case of the Lagrangian action functional is treated). We refer to \cite[Lemma 5.7]{Mer10} for the proof of the second statement.

%%%%%%%%%%%%%%%%%%%%%%%%%%%%%%%%%%%%%%%%%%%%%%%%%%%%%%%%%%
%%%%%%%%%%%%%%%%%%%%%%%%%%%%%%%%%%%%%%%%%%%%%%%%%%%%%%%%%%
%%%%%%%%%%%%%%%%%%%%%%%%%%%%%%%%%%%%%%%%%%%%%%%%%%%%%%%%%%

\section{The action functional near critical points}\label{sec:actloc}

Let $\nu$ be a free homotopy class which is either non-torsion or trivial. Suppose that $\gamma=(x,T)$ is a magnetic geodesic with energy $k$ in the free homotopy class $\nu$ and take $x$ and its iterates as the reference loops in the definition of $S_k$. Let 
\[\pi_{x^n}:(x^n)^*(TM)\longrightarrow \T\]

\noindent be the pull-back bundle of $\pi:TM\rightarrow M$ under the map $x^n:\T\rightarrow M$. There exists an open neighborhood $U$ of the zero-section of $x^*(TM)$ such that 
\[\exp:U\longrightarrow M\, , \ \ \ \ \exp(t,\xi):=\ \exp_{x(t)}(\xi)\]

\noindent is an immersion. Since $x(\T)\neq M$, by shrinking the open set $U$ if necessary, we can assume that $V:=\exp(U)\neq M$ and take a $1$-form $\theta\in\Omega^1(V)$ such that $d\theta=\sigma$ and $\Vert\theta\Vert<\infty$.

For every $n\in \N$ let $U_n$ be an open neighborhood of the zero section of $(x^n)^*(TM)$, such that $\exp(U_n)=\exp(U)=V$. Denote by $H^1(U_n)$ the space of $H^1$-sections of $\pi_{x^n}$ with image contained in $U_n$. Then
\[i_n:H^1(U_n)\times(0,+\infty)\longrightarrow \Lambda_{\nu^n}\, , \ \ \ \ i_n(t,\xi(t),T):=\ (\exp(t,\xi(t)),T)\]

\noindent is a diffeomorphism with an open neighborhood of $\gamma^n$ inside $\Lambda_{\nu^n}$.

Take now $(y,S)$ in this neighborhood and observe that there exists a cylinder $C(x^n,y)$ entirely contained in $V$. Therefore,
\[\int_{C(x^n,y)}\sigma\ =\ \int_{C(x^n,y)}d\theta\ =\ \int_0^1y^*\theta-n\int_0^1x^*\theta\,.\]

\noindent Hence, we can write
\begin{equation}\label{eq:locex}
S_k(y,S)=\overline{S}_k(y,S)-n\int_0^1x^*\theta\,,
\end{equation}

\noindent where $\overline{S}_k:i_n\Big(H^1(U_n)\times (0,+\infty)\Big)\longrightarrow\R$ is defined as 
\begin{equation}\label{eq:locac}
\overline{S}_k(y,S)\ :=\ S\int_0^1\Big[\frac{1}{2S^2}|\dot y(t)|^2+k\Big]dt+\int_0^1y^*\theta\,.
\end{equation}

From this it follows all the local results that hold for the case of exact magnetic fields can be used in our setting, as well. Proposition 1.1 of \cite{AMMP14} applied in this context ensures that 
the mean index of a critical point $\gamma$ of $S_k$ vanishes if and only if the free-period Morse index vanishes for every $n\in \N$.

\begin{lemma}\label{lem:ind}
If $\gamma$ is a critical point for $S_k$, then $\operatorname{ind}(\gamma)=0$ for every $n\in \N$ if and only if $\,\overline{\operatorname{ind}}(\gamma)=0$. 
\end{lemma}

When dealing with degenerate critical points the following lemma (originally due to Gromoll and Meyer) turns out to be useful. We refer to \cite{AMMP14} for the proof.

\begin{lemma}\label{lem:nul}
Let $\gamma=(x,T)$ be a critical point of $S_k$ which corresponds to a non-constant periodic orbit. Then, there is a partition $\N = \N_1\cup \ldots \cup \N_k$, integers $n_1\in \N_1, \ldots, n_k \in \N_k$, and 
$\nu_1,\ldots,\nu_k\in \{0,\ldots, 2d-2\}$ with the following property: $n_j$ divides all the integers in $\N_j$, and $\operatorname{null}(\gamma^n)=\nu_j$ for all $n\in\N_j$.
\end{lemma}

The following tubular neighborhood lemma (see \cite[Section 3]{GM69}), will be needed in the proof of Theorem \ref{iterationofmountainpasses}.

\begin{lemma}\label{lem:ret}
Let $\gamma$ be a critical point of the free-period action functional $S_k$ such that, for some $n\in\N$, we have $\operatorname{ind}(\gamma)=\operatorname{ ind}(\gamma^n)$ and $\operatorname{null}(\gamma)=\operatorname{null}(\gamma^n)$. Then, for every sufficiently small open neighborhood $\mathcal U$ of $\T\cdot\gamma$ there is an open neighborhood $\mathcal V$ of the submanifold $\psi^n(\mathcal U)$ and a smooth map 
\[r:[0,1]\times\mathcal V \longrightarrow \mathcal V\, , \ \ \ \ (t,\beta)\mapsto r_t(\beta)\]

\noindent such that the following hold:

\begin{enumerate}[(i)]
 \item $r_0 = \operatorname{id}$;
\item $r_t|_{\psi^n(\mathcal U)}=\operatorname{id}$, for every $t\in[0,1]$;
\item $r_1(\mathcal V) = \psi^n(\mathcal U)$;
\item $\frac{d}{dt}S_k(r_t(\beta)) < 0$ for all $\beta\in\mathcal V\setminus \psi^n(\mathcal U)$.
\end{enumerate} 
\end{lemma}

We end this section recalling the following result which states that any isolated critical circle of $S_k$ admits a fundamental system of connected open neighborhoods $\mathcal U$ such that the intersection of $\mathcal U$ with the energy sublevel has finitely 
many connected components. Assuming, furthermore, that the critical point has free-period Morse index at least two, one gets that the intersection of such neighborhoods with the energy sublevel is non-empty and connected. We refer to \cite[Lemma 2.5]{AMMP14} for the proof.

\begin{lemma}\label{lem:con}
Let $\T\cdot\gamma$ be an isolated critical circle of the free-period action functional $S_k$ with critical value $c:=S_k(\gamma)$. Then this circle admits a fundamental system of connected open neighborhoods $\mathcal U$ such that 
\[\mathcal U\cap \{S_k < c\}\]

\noindent has finitely many connected components. Moreover, if $\operatorname{ind}(\gamma)\geq2$, for every such neighborhood the intersection 
\[\mathcal U\cap\{S_k < c\}\]
\noindent is non-empty and connected. 
\end{lemma}

%%%%%%%%%%%%%%%%%%%%%%%%%%%%%%%%%%%%%%%%%%%%%%%%%%%%%
%%%%%%%%%%%%%%%%%%%%%%%%%%%%%%%%%%%%%%%%%%%%%%%%%%%%%
%%%%%%%%%%%%%%%%%%%%%%%%%%%%%%%%%%%%%%%%%%%%%%%%%%%%%

\section{Iteration of critical points of \texorpdfstring{$S_k$}{S}}\label{sec:ite}

Here we use the results of Section \ref{sec:actloc} to prove an analogue of \cite[Theorem 2.6]{AMMP14}. Strictly speaking, this shows that critical circles of the functional $S_k$ cannot be 
of mountain pass type if iterated sufficiently many times.

\begin{theorem}
Let $\T \cdot \gamma$ be a critical circle of the free-period action functional $S_k$. Assume that all the iterates of $\gamma$ belong to isolated critical circles of $S_k$. Then, for all $n\in\N$ large enough there exists an open neighborhood $\mathcal W$ of $\T \cdot \gamma^n$ with the following property: if two points $\gamma_0, \gamma_1 \in \{S_k<S_k(\gamma^n)\}$ are contained in the same connected 
component of 
\[\{S_k<S_k(\gamma^n)\} \ \cup \ \mathcal W\, ,\]
\noindent then they are contained in the same connected component of $\{S_k<S_k(\gamma^n)\}$.
\label{iterationofmountainpasses}
\end{theorem}
\begin{proof}
If $\overline{\operatorname{ind}}(\gamma)>0$, then for $n$ large enough, $\operatorname{ind}(\gamma)\geq2$. Therefore the theorem follows directly from Lemma \ref{lem:con}.
Thus we are left with the case $\overline{\operatorname{ind}}(\gamma)=0$, which means, in virtue of Lemma \ref{lem:ind}, that  $\operatorname{ind}(\gamma^n)=0$ for every $n\in \N$. Thanks to Lemma \ref{lem:nul} we get a partition $\N=\N_1\cup\ldots\cup\N_h$ and integers $n_1\in\N_1,\ldots,n_h\in\N_h$, $\nu_1,\ldots,\nu_h\in \{0,...,2d-2\}$ with the property that if $n\in\N_j$, then $n$ is a multiple of $n_j$ and $\operatorname{null}(\gamma^n)=\nu_j$. Without loss of generality let us suppose that $n\in\N_1$.
By Lemma \ref{lem:con} there is a connected neighborhood $\mathcal U$ of $\gamma^{n_1}$ such that 
\[\mathcal U^-:=\ \mathcal U \cap\{S_k<S_k(\gamma^{n_1})\}\]

\noindent has finitely many connected components $\, \mathcal U^-_1,\ldots,\mathcal U^-_r$. We can also assume that on $\mathcal U$ the action functional can be written in the form $S_k=\overline{S}_k+c$, where $\overline{S}_k$ is defined as in \eqref{eq:locac}. For 
every $\alpha\in \{ 1,...,r\}$ choose a loop $\gamma_\alpha\in\mathcal U_\alpha^-$ and for each pair of distinct $\alpha, \beta\in \{1,...,r\}$ choose a path 
\[\Theta_{\alpha\beta}:[0,1]\longrightarrow \mathcal U\]

\noindent connecting $\gamma_\alpha$ and $\gamma_\beta$. By Bangert's technique of ``pulling one loop at a time'' (see \cite[pages 86--87]{Ban80} or \cite[page 421]{Abb13}) applied to $\Theta_{\alpha\beta}$ and the functional $\overline{S}_k$, 
for all sufficiently large multiples $n$ of $n_1$ each path $\psi^{n/n_1}\circ \Theta_{\alpha\beta}$ consisting of iterated loops
is homotopic with fixed endpoints to a path entirely contained in the sublevel set  
\[\big \{\overline{S}_k<\overline{S}_k(\gamma^{n})\big \}\ \subset \ \big \{S_k<S_k(\gamma^{n})\big \}\, .\]

\noindent In other words, $\psi^{n/n_1}(\mathcal U^-)$ is contained in a path-connected component of 
\[\{S_k<S_k(\gamma^{n})\}\, ,\]

\noindent provided $n\in \N_1$ is sufficiently large. Consider such an $n\in \N_1$: applying Lemma \ref{lem:ret} to $\gamma^{n_1}$ and $\gamma^{n}$ we get neighborhoods $\mathcal U'\subseteq \mathcal U$ of $\T\cdot \gamma^{n_1}$ and $\mathcal V'$ of $\T\cdot \gamma^{n}$ and a corresponding deformation retraction $r:\mathcal V'\rightarrow\mathcal V'$ onto $\psi^{n/n_1}(\mathcal U')$ such that
\[\frac{d}{dt} S_k \circ r_t \ \leq \ 0 \, .\]

Set now $\mathcal W$ to be an open neighborhood of $\gamma^{n}$ whose closure is contained in $\mathcal V'$. Let $\gamma_0$ and $\gamma_1$ in $\{S_k<S_k(\gamma^{n})\}$ be as in the statement of the theorem and let 
\[\Gamma:[0,1]\longrightarrow \mathcal W \cup\{S_k<S_k(\gamma^{n})\}\]

\noindent be a path joining them. Consider $s_\pm$ the infimum and the supremum of the times $s\in [0,1]$ such that $\Gamma(s)\in\mathcal W$. Then, $\Gamma(s_\pm)$ belong to 
\[\mathcal V'\cap\{S_k<S_k(\gamma^{n})\}\]

\noindent and the paths $t\mapsto r_{t}(\Gamma(s_\pm))$ are entirely contained in $\{S_k<S_k(\gamma^{n})\}$ by Lemma \ref{lem:ret}. Moreover $r_1(\Gamma(s_\pm))\in\psi^{n/n_1}(\mathcal U^-)$. By what we proved before it follows that $r_1(\Gamma(s_-))$ and $r_1(\Gamma(s_+))$ are in the same connected component of $\{S_k<S_k(\gamma^{n})\}$. Thus, the same is true for $\gamma_0$ and $\gamma_1$.
\end{proof}

%%%%%%%%%%%%%%%%%%%%%%%%%%%%%%%%%%%%%%%%%%%%%%%%%%%%%%%%%%
%%%%%%%%%%%%%%%%%%%%%%%%%%%%%%%%%%%%%%%%%%%%%%%%%%%%%%%%%%
%%%%%%%%%%%%%%%%%%%%%%%%%%%%%%%%%%%%%%%%%%%%%%%%%%%%%%%%%%

\section{Local minimizers on surfaces}\label{sec:min}

From now on, we suppose that $(M,g)$ is an orientable surface of genus at least two. Let $\sigma\in \Omega^2(M)$ be a non-exact 2-form on $M$ and denote by $\phi_t$ the magnetic flow defined by $\sigma$. Changing the orientation of $M$ if necessary, we may suppose that $\sigma$ has positive integral over $M$.
 The following lemma shows that the Ma\~n\'e critical value is finite in this case and, hence, the 
functional $S_k$ introduced in Section \ref{sec:pre} is well-defined on $\Lambda$.

\begin{lemma}\label{manfin}
If $M$ has genus at least two, then $c(g,\sigma)<\infty$. Namely, $\sigma$ admits a bounded primitive on the universal cover.
\end{lemma}
\begin{proof}
Consider first the case of $g_0$, the Riemannian metric with constant curvature $-1$ on $M$ and $\sigma_0$ the associated area form. Take the model 
\[\Big \{(x,y)\in\mathbb R^2\ \Big |\ y>0\Big \}\]
for the universal cover $\widetilde M$ and let $p:\widetilde{M}\rightarrow M$ be the covering map. Then, 
\[p^*g_0\ =\ \frac{dx^2+dy^2}{y^2}\, ,\ \ \ \ p^*\sigma_0\ =\ \frac{1}{y^2}\,dx\wedge dy\, .\]
Define $\displaystyle \widetilde{\beta}_0:=\frac{dx}{y}$ a primitive for $p^*\sigma_0$ and compute
\begin{equation*}
\left \Vert\widetilde{\beta}_0\right \Vert_{p^*g_0}\ =\ \frac{1}{y}\left \Vert dx\right \Vert_{p^*g_0}\ =\ \frac{1}{y}\, y\ =\ 1\, .
\end{equation*}
Hence, $p^*\sigma_0$ admits a bounded primitive on $\widetilde{M}$.

If now $(g,\sigma)$ is any pair defining a magnetic system, with $\int_M\sigma\neq0$, we find a positive constant $A(g)$ and a $1$-form $\beta_\sigma\in\Omega^1(M)$ with 
\begin{equation*}
\Vert v\Vert_{g}\ \geq\  A(g)\cdot \Vert v\Vert_{g_0}\, , \quad \quad \quad \quad \sigma\ =\ \frac{ \int_M\sigma}{2\pi\chi(M)}\, \sigma_0+d\beta_\sigma\, .
\end{equation*}

\noindent If $\widetilde{\beta}_0$ is the primitive of $\sigma_0$ considered before, we have that 
\[\frac{ \int_M\sigma}{2\pi\chi(M)}\, \widetilde{\beta}_0+p^*\beta_\sigma\]

\noindent is a primitive for $\sigma$ and
\begin{align*}
\left \Vert \frac{\int_M\sigma}{2\pi\chi(M)}\widetilde{\beta}_0+p^*\beta_\sigma\right \Vert_{p^*g} &\leq\ \left|\frac{\int_M\sigma}{2\pi\chi(M)}\right|\cdot \left \Vert \widetilde{\beta}_0\right \Vert_{p^*g}\!\!\! +\ \left \Vert p^*\beta_\sigma\right \Vert_{p^*g}\\
&\leq \ \left|\frac{\int_M\sigma}{2\pi\chi(M)}\right|\cdot \left \Vert \frac{1}{A(g)}\, \widetilde{\beta}_0\right \Vert_{p^*g_0}\!\!\!\!\! +\ \left \Vert \beta_\sigma\right \Vert_g\ <\ \infty\, .\qedhere
\end{align*}
\end{proof}
\bigskip

In the remainder of this section we prove that, if the energy value $k$ is sufficiently small, then there exists a closed magnetic geodesic of energy $k$ which is a local minimizer of $S_k$, provided that $\sigma$ is a so-called \textit{oscillating form}, i.e.\ that there is a point $x\in M$ such that $\sigma_x<0$.

Let us start by defining the family of Taimanov functionals $\mathcal T_k:\mathcal F_+\rightarrow \R$, where $k\in(0,+\infty)$ and $\mathcal F_+$ is the space of positively oriented (possibly with boundary and not necessarily connected) embedded surfaces in $M$ (in
\cite{Tai92b,Tai92a,Tai93}, Taimanov considered the so-called \textit{films}):
\begin{equation}
\mathcal T_k(\Pi) := \ \sqrt{2k}\cdot l(\partial \Pi) + \int_\Pi \sigma\,,
\label{Taimanovfunctional}
\end{equation}
where $l(\partial \Pi)$ denotes the length of the boundary of $\Pi$. We readily compute for future applications (notice that $\emptyset \in \mathcal F_+$) 
\begin{equation}
\mathcal T_k (\emptyset) \ = \ 0 \, , \ \ \ \ \mathcal T_k(M) \ = \ \int_M \sigma \ >\ 0\, .
\label{positivi}
\end{equation}

Observe that the family $\{\mathcal T_k\}$ is increasing in $k$ and each $\mathcal T_k$ is bounded from below since
\begin{equation*}
\mathcal T_k(\Pi)\ \geq \ -\Vert \sigma\Vert_\infty\cdot \operatorname{area}_g(M)\, .
\end{equation*}

Define the value
\begin{equation}\label{taival}
\tau_+(M,g,\sigma):= \ \inf\big \{ k\ \big |\ \inf \mathcal T_k\geq 0\ \big\} \ = \ \sup \big \{ k\ \big  |\ \inf \mathcal T_k< 0\ \big\}\, .
\end{equation}
The functionals $\mathcal T_k$ can be lifted to any finite cover $p':M'\rightarrow M$ giving rise to the set of values $\tau_+(M',g,\sigma)$.
%Although it is not yet clear if 
%\[\tau_+(M',g,\sigma) \ \leq \ \tau_+(M'',g,\sigma)\]
%when $p':M''\rightarrow M'$ is a finite cover
We can define the \textit{Taimanov critical value} as
\begin{equation}
\tau_+(g,\sigma):=\ \sup\Big\{\, \tau_+(M',g,\sigma)\ \Big |\ p':M'\rightarrow M \mbox{ finite cover }\Big\}\, .
\end{equation}

In \cite{CMP04} it was shown that, when $\sigma$ is exact, $\tau(g,\sigma)=c_0(g,\sigma)$, where $c_0(g,\sigma)$ denotes the Ma\~n\'e critical value of the abelian cover of $M$ (here we write $\tau$ instead of $\tau_+$, since in the 
exact case one does not have to struggle with the orientation of the embedded submanifolds). When $\sigma$ is not exact, $\tau_+(g,\sigma)$ is positive exactly when $\sigma$ is oscillating as the next lemma shows. In order to have a more precise representation of the Taimanov critical value, let us define the auxiliary value $c_+(g,\sigma)$

\vspace{-4mm}

\begin{align*}
c_+(g,\sigma)\ &:\,=\ \inf \ \Big \{k\in \R\ \Big |\ \forall \ \gamma \in \Gamma\,,\ \widetilde {S}_k(\tilde \gamma) \geq 0\, \Big \}\\ 
                        &\ =\ \sup \,\Big \{k\in \R\ \Big |\ \exists \ \gamma \in \Gamma\, \text{ s.t. } \widetilde{S}_k(\tilde\gamma)<0\Big \}\,,
\end{align*}
where $\Gamma$ is the class of positively oriented simple loops on the universal cover of $M$. For the proof of the main theorem we need to work with energy levels below $\tau_+(g,\sigma)$ that are also smaller than $c(g,\sigma)$. Since we do not know if, in general, there is an inequality between these two quantities, we simply set 
\[\tau^*_+(g,\sigma):=\ \min\ \Big \{\tau_+(g,\sigma),c(g,\sigma)\Big \}\, .\]

\vspace{-4mm}

\begin{lemma}\label{lem:tau}
Suppose $\sigma\in\Omega^2(M)$ is a non-exact 2-form on $M$ with positive integral. Then $c_+(g,\sigma)>0$ if and only if $\sigma$ is oscillating. Furthermore,
\[\tau_+^*(g,\sigma)\ \geq \ c_+(g,\sigma)\, .\]
\end{lemma}

\begin{proof}
If $\sigma$ is everywhere non-negative, then $\widetilde{S}_k(\tilde\gamma)\geq0$ for every $k>0$ and every $\tilde\gamma\in\Gamma$. Conversely, assume that $\tilde \sigma$ is negative at some point $\tilde{x}$. Then there exists a small positively oriented disc $\tilde D$ around $\tilde{x}$ such that 
\[\int_{\tilde D}\tilde{\sigma}\ <\ 0\, .\]
Given $k>0$, we define $\tilde{y}_k:[0,l(\partial{\tilde D})]\rightarrow \widetilde{M}$ to be the parametrization of $\partial{\tilde D}$ by $\sqrt{2k}$-times the arc length. A simple computation shows that 
\[\widetilde{S}_k(\tilde y_k)\ =\ \sqrt{2k}\cdot l(\partial{\tilde D})+\int_{\tilde D}\tilde{\sigma}\, .\]

\noindent Hence, for $k$ sufficiently small $\widetilde{S}_k(\tilde y_k)<0$. This shows that  $c_+(g,\sigma)>0$. We now prove that $\tau_+^*(g,\sigma)\geq c_+(g,\sigma)$. Clearly $c(g,\sigma)\geq c_+(g,\sigma)$. Thus, let us show that $\tau_+(g,\sigma)\geq c_+(g,\sigma)$. 
Suppose $k<c_+(g,\sigma)$. Then there exists a positively oriented simple loop $\tilde\gamma$ on the universal 
cover of $M$ such that $\widetilde{S}_k(\tilde\gamma)<0$. Now since $M$ is a surface, $\pi_1(M)$ is \textit{residually finite} (i.e.\ the intersection of all its normal subgroups with finite index is trivial; see \cite{Hem72} or \cite{Bau62}) and therefore we can find a finite cover $p':M'\rightarrow M$ such that the projection $\gamma'=p''(\tilde\gamma)$ to $M'$ is simple, where $p'':\widetilde M\rightarrow M'$ (see \cite{CMP04} for a similar argument where the Abelian cover is considered instead). Therefore, $\gamma'$ bounds an embedded disc $D'$ on $M'$ (actually $D'=p''(\tilde D)$) and there holds
\[S_k (\gamma')\ =\ \widetilde{S}_k(\tilde\gamma)\ <\ 0\, .\]
From the inequality $\, \sqrt{2k}\, \ell \leq kT + \frac{1}{2T}\ell^2$ for all $T>0$ it follows that 
\[\mathcal T_k(D')=\sqrt{2k}\cdot l(\gamma') + \int_{D'} \sigma \ \leq \ S_k(\gamma')\ <\ 0\,,\]
which shows that $k<\tau_+(g,\sigma)$.
\end{proof}

%\begin{oss}
%Fix $\nu$ a free homotopy class and consider $\mathcal C^\nu_+$ the set of embedded (maybe also immersed is possible) cylinders $C(s,t):[0,1]\times \mathbb T\rightarrow M$, where $C(s,\cdot)$ is an element of $\nu$ and $C$ preserves the orientation (on $C$ the standard orientation is given). The action $S_k$ of an element in $\mathcal C^\nu_+$ is well-defined.
%Put $k^\nu_+(M,g,\sigma):=\inf\{ k\ |\ S_k\big|_{\mathcal C^\nu_+}>0\}$ and $k_+(M,g,\sigma):=\sup_\nu k^\nu_+(M,g,\sigma)$.
%Then, you take $k_+(g,\sigma)=\sup\{k_+(M',g,\sigma)\ |\ M' \mbox{ a finite cover of }M\}$. My guess is that $k_+(g,\sigma)=\tau_+(g,\sigma)$. Moreover, $k_+(g,\sigma)$ could be defined using the action of embedded positively oriented rectangles $R$ in the universal cover with vertices $\widetilde{x}, \widetilde{y},g(\widetilde{x}),g(\widetilde{y})$, where $\widetilde{x},\widetilde{y}$ are any two points on $\widetilde{M}$ and $g$ is any deck transformation of $\widetilde{M}$.   
%\end{oss}

\vspace{5mm} 

We can now state Taimanov's main theorem about the existence of global minimizers for $\mathcal T_k$. It is worth to point out that this result does not depend on the genus of the surface and holds more generally for any oscillating non-exact $2$-form. 
For the proof we refer to \cite{Tai92a} for the case $M=S^2$ and to \cite{Tai93} for the general case. 
The reader can also take a look at \cite{CMP04} for a new proof using methods of geometric measure theory.

\begin{theorem}[Taimanov, 1992]
Let $(M,g)$ be an oriented surface and let $\sigma$ be an oscillating non-exact 2-form with positive integral over $M$. For every $k<\tau_+(g,\sigma)$ there is a smooth positively oriented embedded surface $\Pi$ which is a global minimizer of $\mathcal T_k$ on the space of positively oriented surfaces on a finite cover $M'$ with $\mathcal T_k(\Pi)<0$. Each boundary component of $\Pi$ is then a simple closed magnetic geodesic for the flow defined by $\sigma$ on $M'$.
\label{Taimanov92}
\end{theorem}

Now we show how, in the case $k<\tau_+(g,\sigma)$, from this result we can get the existence of a closed magnetic geodesic which is a local minimizer for $S_k$.

\begin{lemma}
Suppose $k<\tau_+(g,\sigma)$. Then, there is a closed magnetic geodesic $\gamma$ with energy $k$ which is a local minimizer of $S_k$ in $\Lambda$. Furthermore, if $\gamma$ is contractible, we may suppose that $S_k(\gamma)<0$.
\label{localminimizers}
\end{lemma}

\begin{proof}
If $k<\tau_+(g,\sigma)$, then by Theorem \ref{Taimanov92} we get the existence of a positively oriented embedded surface $\Pi$ which is a global minimizer of $\mathcal T_k$ in the space of positively oriented embedded surfaces on $M'$ and such that $\mathcal T_k(\Pi)<0$. 
By (\ref{positivi}) it follows that $\Pi$ has non-empty boundary. Each of its boundary components is then a simple magnetic geodesic for the lifted flow. Take a boundary component $\gamma$ and let $y$ be a simple closed curve which is $C^1$-close to $\gamma$. Consider 
the surface $\Pi_y$ obtained from $\Pi$ by changing the boundary component $\gamma$ with $y$. Since $\Pi$ is a global minimizer of $\mathcal T_k$ among the space of positively oriented embedded surfaces, we get
\begin{equation}
\mathcal T_k (\Pi_y)\ \geq \ \mathcal T_k(\Pi)\, .
\label{minimizer}
\end{equation}
Observe that 
\begin{align*}
\mathcal T_k(\Pi_y) - \mathcal T_k(\Pi) &=\ \sqrt{2k} \cdot l(\partial \Pi_y) + \int_{\Pi_y} \sigma -\sqrt{2k}\cdot l(\partial \Pi) - \int_\Pi \sigma \\
&=\ \sqrt{2k}\cdot \Big [l(y)-l(\gamma)\Big ] + \int_{\Pi_y\setminus \Pi} \sigma\,.
\end{align*}
Hence,
\begin{align*}
S_k(y)-S_k(\gamma) &\stackrel{\mbox{}^{(\star)}}{\geq} \ \sqrt{2k} \cdot \Big [ l(y)-l(\gamma)\Big ] + \int_{C(y)} \sigma - \int_{C(\gamma)} \sigma \\
&\stackrel{\mbox{}^{(\star\star)}}{=} \sqrt{2k} \cdot \Big [ l(y)-l(\gamma)\Big ] + \int_{\Pi_y\setminus \Pi}\!\! \sigma\ = \ \mathcal T_k(\Pi_y)-\mathcal T_k(\Pi)\, ,
\end{align*}
where in the inequality $(\star)$ we have used again the elementary estimate 
\[\frac{1}{2T} \, \ell^2 + kT \ \geq \ \sqrt{2k}\, \ell\]

\noindent and the fact that 
\[S_k(\gamma) \ = \ \sqrt{2k}\, l(\gamma) + \int_{C(\gamma)} \sigma\]

\noindent since $\gamma$ has constant speed, being a magnetic geodesic. The equality $(\star\star)$ is a consequence of the fact that the form $\sigma$ is exact on each 2-torus.

Therefore, it follows from (\ref{minimizer}) that each boundary component of $\Pi$ is a local minimizer for the functional $S_k$ in $C^1(\T,M')\times (0,+\infty)$. Since the projection map 
\[C^1(\T,M')\times (0,+\infty) \ \longrightarrow C^1(\T,M)\times (0,+\infty)\]
is open, the projection of each boundary component of $\Pi$ defines a closed magnetic geodesic on $M$ which is a local minimizer for $S_k$ on the space $C^1(\T,M)\times (0,+\infty)$. Notice that these closed magnetic geodesics might not be simple after the projection.

We observe now that (strict) local minimizers in $C^1(\T,M)\times (0,+\infty)$ are also (strict) local minimizers in $H^1(\T,M)\times (0,+\infty)=\Lambda$. This follows from \cite[Lemma 2.1]{AMP13}, since, as we have seen in Section \ref{sec:actloc}, locally $S_k$ reduces to a classical Lagrangian action functional. Hence, we conclude that $\gamma$ is a minimizer of $S_k$ on $\Lambda$.

From the discussion above we see that the lemma is proven if we can pick $\gamma$ to be non-contractible. If this is not possible, it means that the boundary components of all global minimizers of $\mathcal{T}_k$ on $M'$ are contractible curves. Pick one of these global minimizers $\Pi$ with the minimal number, say $n$, of boundary components; notice that $n\geq 1$ by (\ref{positivi}). Recall that in the contractible case we made the natural choice of taking a constant loop as the reference loop in the definition of $S_k$. Denote by $\gamma_1,\ldots,\gamma_n$ the connected components of the boundary of $\Pi$ parametrized by $\sqrt{2k}$ times the arc-length. Let $\{D_1,\ldots,D_n\}$ be the closed embedded discs bounded by the $\gamma_i$'s respectively. Clearly there exists an $i_*$ such that one of the two alternative holds:

\begin{enumerate}
\item $D_{i_*}\subseteq \Pi$ ($D_{i_*}$ is positively oriented), or
\item $D_{i_*}\cap\Pi=\gamma_{i_*}$ ($D_{i_*}$ is negatively oriented).
\end{enumerate}
In the first case,  $\Pi\setminus D_{i_*}$ is a (possibly empty) positively oriented embedded surface and
\[\mathcal T_k(\Pi)\ =\ \mathcal T_k(\Pi\setminus D_{i_*})\ +\ S_k(\gamma_{i_*})\,.\]
Since $\Pi$ is an \textit{absolute} minimizer, we get that $S_k(\gamma_{i_*})\leq0$. We claim that, actually, $S_k(\gamma_{i_*})<0$. Otherwise, assume by contradiction $S_k(\gamma_{i_*})=0$, then $\Pi\setminus D_{i_*}$ is also a global minimizer for $\mathcal T_k$  but its boundary has one connected component less than $\Pi$, which is a contradiction.

In the second case, let $\overline{D}_{i_*}$ be the disc $D_{i_*}$ taken with the opposite orientation. Then, $\Pi\cup \overline{D}_{i_*}$ is a positively oriented embedded surface and
\[\mathcal T_k(\Pi)\ =\ \mathcal T_k(\Pi\cup \overline{D}_{i_*})\ +\ S_k(\gamma_{i_*})\,.\]
As before, we conclude that $S_k(\gamma_{i_*})<0$.

In both cases, projecting to $M$ if necessary we get the existence of a local minimizer of $S_k$ with negative action and the lemma follows.
\end{proof}
\medskip

We conclude this section mentioning the fact that, since $M$ is an orientable surface, a closed curve in $M$ which is a (strict) local minimizer of $S_k$ on $\Lambda$ remains a (strict) local minimizer also when iterated.
Counterexamples to this statement for geodesic flows in dimension bigger than two or on non-orientable surfaces are described in \cite{Hed32} and \cite[Example 9.7.1]{KH95}, respectively. 

\begin{lemma}
If $\alpha$ is a (strict) local minimizer of $S_k$ on $\Lambda$, then for every $n\geq 1$ the $n$-th iterate $\alpha^n$ is also a (strict) local minimizer of $S_k$ on $\Lambda$.
\label{persistenceoflocalminimizers}
\end{lemma}

The proof in \cite[Lemma 3.1]{AMP13} for the case of the Lagrangian action functional goes through without any change, thanks to the discussion in Section \ref{sec:actloc}.

%%%%%%%%%%%%%%%%%%%%%%%%%%%%%%%%%%%%%%%%%%%%%%%%%%%%%%%%%%
%%%%%%%%%%%%%%%%%%%%%%%%%%%%%%%%%%%%%%%%%%%%%%%%%%%%%%%%%%
%%%%%%%%%%%%%%%%%%%%%%%%%%%%%%%%%%%%%%%%%%%%%%%%%%%%%%%%%%

\section{The contact property}\label{cha:con}

In the next two sections we will complete the proof of the Main Theorem. A decisive step will be to prove that some Palais-Smale sequence $(x_n,T_n)$ has an accumulation point (see Lemma \ref{lem:str}). As we have seen in Theorem \ref{teo:pal}, the crucial issue is that, in general, there might exist Palais-Smale sequences such that $T_n\rightarrow+\infty$. In the exact case, Contreras showed how this problem is related with the contact geometry of the level set \cite[Proposition F]{Con06}. Namely, he proved that every Palais-Smale sequence with energy $k$ has bounded period, provided the energy level $E^{-1}(k)$ is of contact type. Since being of contact type is an open condition, showing that $k\in(0,c(L))$ is an energy value of contact type allows to prove that there are infinitely many magnetic geodesics with \textit{every} fixed energy close to $k$.

In light of a possible generalization of Contreras' result to our setting, it is natural to try to find out which energy levels are of contact type, when the magnetic form is oscillating. However, so far only few facts are known in this case. The only positive result is Corollary 4.14 in \cite{Ben14} (see also \cite{Pat09}), which asserts that $E^{-1}(k)$ is of positive contact type for all large enough $k$. On the other hand, obstructions to the contact property are obtained by computing the action of the Liouville measure. Following \cite{Pat09}, there is an energy value $c_h(g,\sigma)$ called the \textit{helicity} such that
\begin{itemize}
	\item the inequality $0<c_h(g,\sigma)\leq c(g,\sigma)$ holds,
  \item the action of the Liouville measure on $E^{-1}(k)$ is positive for $k> c_h(g,\sigma)$ and negative for $k< c_h(g,\sigma)$.
\end{itemize}
Hence, $E^{-1}(k)$ cannot be of \textit{negative} contact type for $k\geq c_h(g,\sigma)$ and it cannot be of \textit{positive} contact type for $k\leq c_h(g,\sigma)$.

Since we are looking at the energy range $(0,\tau^*_+(g,\sigma))$, we could hope that low energy levels are of negative contact type, as happens when $\sigma$ is symplectic \cite{Pat09,Ben14}. However, the next proposition shows that, when the magnetic form is oscillating, low energy levels are never of \textit{negative} contact type.
\begin{proposition}
Let $(M,g,\sigma)$ be an oscillating magnetic system on an orientable Riemannian surface with genus at least two. For every small enough $k>0$ the energy level $E^{-1}(k)$ is not of negative contact type. As a corollary, $E^{-1}(k)$
is not of contact type. \label{nocontact}
\end{proposition}

\begin{proof}
Define $\overline{\sigma}:=K\mu\in\Omega^2(M)$, where $K$ is the Gaussian curvature and $\mu$ is the area form of $g$. There exists $\beta\in\Omega^1(M)$ such that $\sigma=a_\sigma\overline{\sigma}+d\beta$, where
\[
a_\sigma\ :=\ \frac{\int_M\sigma}{\int_M\overline{\sigma}}\ =\ \frac{\int_M\sigma}{2\pi\chi(M)}\ <\ 0
\]

\noindent by the Gauss-Bonnet Theorem. Following the computation in \cite[Proposition 4.5]{Ben14}, the action of a null-homologous $\phi_t$-invariant measure $\zeta$ on $E^{-1}(k)$ is given by

\[
\mathcal A_k(\zeta)\:=\ \int_{E^{-1}(k)}\Big[2k+\beta_x(v)+a_\sigma f(x)\Big]\, d\zeta\,,
\]
where $f:M\rightarrow \R$ is defined by the equation $\sigma=f\mu$. Since $\sigma$ is oscillating, $\min f<0$.

The energy level $E^{-1}(k)$ is not of negative contact type, provided there exists a null-homologous $\phi_t$-invariant measure $\zeta$ on $E^{-1}(k)$ with $\mathcal A_k(\zeta)\geq0$. 
Denote by 
\[C^f_c:=\ \big \{x\in M\ \big |\ f(x)\leq c\big \}\]

\noindent the closed sublevels of $f$ and let $b\in(\min f,0)$ be a regular value of $f$. Construct now a function $f_*:M\rightarrow(-\infty,0)$ such that $f_*$ coincides with $f$ on the sublevel $C^f_b$ and $C^{f_*}_b=C^f_b$. Since $f_*$ is a negative function, we can define the Ginzburg action functional $G^{f_*}_k:E^{-1}(k)\rightarrow \R$ for all sufficiently small $k>0$ (see \cite{Gin87} for further details). This functional has the property that its critical points are the support of those periodic orbits $\gamma$ such that
\begin{equation}\label{leng}
l(\pi(\gamma))\ \leq \ C\sqrt{2k}\,, 
\end{equation}

\noindent for some $C>0$ independent of $k$. As is proved in \cite{Ben14} we can expand this functional in the parameter $k$ around zero getting
\[
G^{f_*}_k(x,v)\ =\ 2\pi-\frac{2\pi}{f_*(x)}k+o(k)\,.
\]

For every $k$ let now  $(x_k,v_k)$ be an absolute maximizer for $G^{f_*}_k$. By the expansion above for every $\varepsilon>0$ there exists $k_\varepsilon$ such that 
\[d\left (x_k,C^{f_*}_{\min f_*}\right )\ <\ \varepsilon\, , \ \ \ \text{for} \ k<k_\varepsilon\, .\]

\noindent Call $\gamma_k$ the periodic orbit through $(x_k,v_k)$. By Inequality \eqref{leng}, there is $k'_\varepsilon\leq k_\varepsilon$ such that $\pi(\gamma_k)\subset C^{f_*}_b=C^f_b$ if $ k<k'_\varepsilon.$
If $\zeta_k$ is the invariant measure supported on $\gamma_k$, then

\[
\mathcal A_k(\zeta_k)\ = \int_{E^{-1}(k)}\Big[2k+\beta_x(v)+a_\sigma f(x)\Big]\, d\zeta_k\ \geq \ 2k-\Vert\beta\Vert\sqrt{2k}+a_\sigma\cdot b.
\]
Since $a_\sigma\cdot b>0$, $\mathcal A_k(\zeta_k)\geq0$ for $k$ sufficiently small.
\end{proof}
\bigskip

In view of the previous proposition, we will really need Struwe's argument, contained in Lemma \ref{lem:str}, to prove the existence of converging Palais-Smale sequences (even if only on \textit{almost every} energy level).

%%%%%%%%%%%%%%%%%%%%%%%%%%%%%%%%%%%%%%%%%%%%%%%%%%%%%%%%%%
%%%%%%%%%%%%%%%%%%%%%%%%%%%%%%%%%%%%%%%%%%%%%%%%%%%%%%%%%%
%%%%%%%%%%%%%%%%%%%%%%%%%%%%%%%%%%%%%%%%%%%%%%%%%%%%%%%%%%

\section{The minimax values}\label{sec:minmax}

In this section we follow \cite[Section 3]{AMMP14} and define the minimax values which allow us to prove the existence of infinitely many closed magnetic geodesics on almost every energy level below $\tau_+^*(g,\sigma)$. For the reader's convenience we keep the same notation as in the main reference.

For any $k\in (0,\tau^*_+(g,\sigma))\subseteq (0,c(g,\sigma))$, let $\alpha_k\in \Lambda$ be a local minimizer of $S_k$, whose existence
is guaranteed by Lemma \ref{localminimizers}. Recall that if $\alpha_k$ is contractible we may suppose $S_k(\alpha_k)<0$.
For the rest of the paper let us fix $k_*\in (0,\tau_+^*(g,\sigma))$. Without loss of generality we may assume that $\alpha_{k_*}$ is a strict local minimizer for $S_{k_*}$, as otherwise there trivially exist infinitely many geometrically distinct closed magnetic geodesics with energy $k_*$. By Lemma \ref{persistenceoflocalminimizers} all iterates $\alpha_{k_*}^n$ are also strict local minimizers. Denote by $\nu_*$ the free homotopy class of $\alpha_{k_*}$. Thanks to Lemma \ref{lem:equi} and to the fact that on a surface of 
genus at least two a free-homotopy class is either the trivial class or has infinite order, we can take $S_k$ to be $\N$-equivariant on $\bigcup_{n\in\N}\Lambda_{\nu^n_*}$. 
\medskip

Since $k_*$ is strictly smaller than $c(g,\sigma)$, by Lemma \ref{lem:unb} the functional $S_{k_*}$ is unbounded from below on any connected component of $\Lambda$ and hence there exists an element $\mu\in\Lambda_{\nu_*}$ such that 
\[S_{k_*} (\mu) \ <\ S_{k_*}(\alpha_{k_*})\, .\]

Thanks to Lemma \ref{lem:strict} we can find a bounded open neighborhood $\mathcal U\subset\Lambda_{\nu_*}$ of $\T\cdot \alpha_{k_*}$ such that
\begin{itemize}
\item its closure intersects the critical set of $S_{k_*}$ only in $\T\cdot \alpha_{k_*}\,$, 
\item the inequality $\displaystyle\inf_{\partial \mathcal U} \ S_{k_*}  \ > \ S_{k_*}(\alpha_{k_*})\,$ holds,
\item the period $T$ is bounded away from zero for all $(x,T)\in \mathcal U\,$.
\end{itemize}

For every $k\in (0,\tau_+^*(g,\sigma))$ let $M_k$ be the closure of the set of local minimizers of $S_k$ contained in $\mathcal U$. All the elements of $M_k$ are critical points for $S_k$ and they are also strict local minimizers, provided $M_k$ is a finite union of critical circles. The proof of the following lemma is analogous to the one of \cite[Lemma 3.1]{AMMP14}.

\begin{lemma}
There exists a closed interval $J=J(k_*)\subset (0,\tau_+^*(g,\sigma))$ whose interior contains $k_*$ and which has the following properties:

\begin{enumerate}
\item For every $k\in J$ the set $M_k$ is a non-empty compact set.

\item For every $k\in J$ there holds 
\[S_k(\mu) \ <\ \min_{M_k} \ S_k\, .\]

\item For every $k\in J$ there holds 
\[\sup_{k'\in J} \ \max_{M_{k'}} \ S_k \ < \ \inf_{\partial \mathcal U} \ S_k\, ;\]

if $\alpha_{k_*}$ is contractible, there also holds $ \displaystyle \ \ \sup_{k'\in J} \ \max_{M_{k'}} \ S_k < 0\,.$
\end{enumerate}
\label{lemma6.1}
\end{lemma}

\vspace{2mm}

For every $n\in \N$ and every $k\in J$ we define
\[\mathcal P_n (k) := \ \Big \{ u\in C^0([0,1],\Lambda) \ \Big |\ u(0)\in \psi^n(M_k)\, ,\ u(1)=\mu^n\Big \}\]

\noindent as the set of continuous paths in $\Lambda$ which join the $n$-th iterate of some element in $M_k$ with the $n$-th iterate of $\mu$. They yield the minimax functions  
\[c_n:J\longrightarrow \R\, , \ \ \ \ c_n(k) := \ \inf_{u\in \mathcal P_n(k)}\ \max_{s\in [0,1]} \ S_k(u(s))\, .\]
Applying point \textit{(2)} of the previous lemma, we have 
\begin{equation}
c_n(k) \ \geq \ \min_{\psi^n(M_k)} \ S_k \ = \ n\cdot \min_{M_k} \ S_k \ >\ n\cdot S_k(\mu)\, .
\label{6.3}
\end{equation}

As already discussed in the introduction, in order to overcome the possible lack of the Palais-Smale condition we will use the fact that each of the minimax functions $c_n$ is monotone. Strictly speaking, we use the fact that a monotone function is differentiable at almost every point to construct bounded Palais-Smale sequences at almost every level $k$ (namely at each level where all the functions $c_n$ are differentiable). Unfortunately, this monotonicity property does not follow directly from the monotonicity of the functional $S_k$, since the minimax class $\mathcal P_n(k)$ does depend on $k$. However, monotonicity can be shown exactly as done in \cite{AMMP14} for the exact case and, hence, we omit the proof.

\begin{lemma}
For every $n\in \N$, $c_n$ is monotonically increasing on $J$.
\end{lemma}

\vspace{2mm}

If the strict local minimizer $\alpha_k$ is contractible, we need to know more about the behavior of the minimax functions $c_n$. The next lemma is a simple generalization of \cite[Lemma 3.4]{AMMP14} and is based on the so-called 
``technique of pulling one loop at a time'' introduced first by Bangert \cite{Ban80}. Observe that we can run the same proof as in \cite[Lemma 2.5]{AMP13} since 
\[S_k\big|_{\Lambda_0}\circ p\ =\ \widetilde S_k\, ,\]

\noindent where $p:\widetilde {\Lambda} \rightarrow \Lambda$ denotes the projection map. Indeed, no assumption on the compactness of $M$ is required in that argument. 

\vspace{5mm}

\begin{lemma}
Let $K_0,K_1 \subseteq \Lambda_0$ be compact sets and let 
\[\mathcal R_n := \ \Big \{u\in C^0([0,1],\Lambda_0)\ \Big |\ u(0)\in \psi^n (K_0)\, ,\ u(1)\in \psi^n(K_1)\Big \}\, .\]
Fix a number $k$ and set 
\[c_n:= \ \inf_{u\in \mathcal R_n} \max_{s \in [0,1]}\ S_k(u(s))\, .\]
Then, there exists a number $A\in \R$ such that 
\[c_n \ \leq \ n\cdot \max_{K_0\cup K_1} \ S_k \ + \ A \, ,\quad\forall \ n\in \N\, .\]
\label{Bangert80}
\end{lemma}
We apply now this lemma to 
$\hat k =\max J$, $K_0=M_{\hat k}$ and $K_1=\{\mu\}$ and use the fact that $S_{\hat k}(\mu)$ and $\max_{M_{\hat k}} S_{\hat k}$ are negative (see Lemma \ref{lemma6.1}) to obtain
\[c_n(\hat k) \ \longrightarrow \ -\infty \ \ \ \ \text{for} \ \ n\longrightarrow +\infty\, .\]

\noindent In particular, since each $c_n$ is monotonically increasing in $k$, we get 
\begin{equation}
\lim_{n\rightarrow +\infty} \ c_n \ = \ -\infty \ \ \ \ \ \text{uniformly on} \ J\, .
\label{uniformity}
\end{equation}

In the non-contractible case, a simple topological argument will play the same role as \eqref{uniformity} in the proof of the Main Theorem (see the last section for further details).
%%%%%%%%%%%%%%%%%%%%%%%%%%%%%%%%%%%%%%%%%%%%%%%%%
%%%%%%%%%%%%%%%%%%%%%%%%%%%%%%%%%%%%%%%%%%%%%%%%%
%%%%%%%%%%%%%%%%%%%%%%%%%%%%%%%%%%%%%%%%%%%%%%%%%

\section{Proof of the main Theorem}\label{cha:fin}

Here we complete the proof of the Main Theorem following Section 3.3 in \cite{AMMP14}. By the discussion in the previous sections this will be an easy consequence of Theorem \ref{teoremafinale} below. 

We saw in Theorem \ref{Palaissmale} that Palais-Smale sequences for the functional $S_k$ at level $c$ in the same free homotopy class and with period bounded from above have a limit point, provided:

\begin{enumerate}
\item  They are not contractible, or
\item  They are contractible and $c\neq 0$.
\end{enumerate}

The next lemma shows that sequences satisfying such bound exist on almost every energy level close to $k_*$. The original argument is due to Struwe \cite{Str90} (see also \cite[Proposition 7.1]{Con06}). 
The formulation we give here is a generalization of Lemma 3.5 in \cite{AMMP14} to the weakly exact case and the proof contained therein goes through word by word. Let $c_n:J\rightarrow \R$ be the 
sequence of minimax functions defined in the previous section. In case the strict local minimizer $\alpha_{k_*}$ is contractible, we know by (\ref{uniformity}) that there exists $n_0\in \N$ such that all $c_n$'s are negative for $n\geq n_0$. 

\begin{lemma}\label{lem:str}
Let $k$ be an interior point of $J$ at which $c_n$ is differentiable and such that the set $M_k$ is a finite union of critical circles. If $\alpha_{k_*}$ is contractible suppose in addition that $n\geq n_0$. Then, for every open neighborhood $\mathcal V$ of 
\begin{equation*}
 \operatorname{crit} S_{k}\ \cap\  \big \{S_{k}=c_n(k)\big \}\, ,
\end{equation*}

\noindent there exists an element $v$ of $\mathcal P_{n}(k)$ such that $S_{k}(v(0))<c_n(k)$ and
\begin{equation*}
v([0,1])\ \subset \ \big \{S_{k}< c_n(k)\big \}\cup \mathcal V.
\end{equation*}

\noindent In particular $c_n(k)$ is a critical value of $S_{k}\big|_{\Lambda_{\nu_*^n}}$.
\label{lemma7.1}
\end{lemma}

\vspace{3mm}

\begin{theorem}
If the local minimizer $\alpha_{k_*}$ is strict, then the energy value $k_*$ has a neighborhood $J\subseteq (0,\tau_+^*(g,\sigma))$ such that for almost every $k\in J$ the energy level $E^{-1}(k)$ has infinitely many periodic orbits freely homotopic to some iteration of $\alpha_{k_*}$.
\label{teoremafinale}
\end{theorem}

\begin{proof}
Let $J=J(k_*)$ and  $c_n:J\rightarrow \R$ be as above. Since the countably many functions $c_n$ are monotone, it follows from the Lebesgue Theorem that the set of points $J'$ at which all the $c_n$ are differentiable has full measure in $J$. We shall prove that for every $k\in J'$ the energy level $E^{-1}(k)$ has infinitely many geometrically distinct periodic orbits.

Thus, pick $k\in J'$. If $M_k$ consists of infinitely many critical circles, then there is nothing to prove. Therefore, we may suppose that $M_k$ consists of only finitely many critical circles, which are local minimizers. Assume by contradiction that $E^{-1}(k)$ has only finitely many periodic orbits homotopic to some iteration of $\alpha_{k_*}$. Then, the critical set of $S_k$ restricted to $\bigcup_{n\in\N}\Lambda_{\nu^n_*}$ consists of finitely many critical circles 
\[\T \cdot \gamma_1\, , \ \T\cdot \gamma_2\, , \ldots \, , \ \T\cdot \gamma_l\, ,\]

\noindent together with their iterates $\T\cdot \gamma_j^n$, for $1\leq j\leq l$ and for every $n\in \N$. By Theorem \ref{iterationofmountainpasses} we can find a natural number $n_*$ such that the following is true: for every $n\geq n_*$ and 
for every $1\leq j \leq l$ there exists a neighborhood $\mathcal W_{j,n}$ of $\T\cdot \gamma_j^n$ such that any two points in $\{S_k<S_k(\gamma_j^n)\}$ which can be connected within 
\[\big \{S_k < S_k(\gamma_j^n) \big \} \cup \mathcal W_{j,n}\]

\noindent can be also connected in $\{S_k<S_k(\gamma_j^n)\}$. Moreover, the sets $\mathcal W_{j,n}$ can be chosen to be so small that their closures are pairwise disjoint. If $\alpha_{k_*}$ is not contractible we can find $n\in\N$ such that $\alpha_{k_*}^n$ is not freely homotopic to $\gamma_j^m$, for every $1\leq j\leq l$ and $1\leq m\leq n_*-1$. If $\alpha_{k_*}$ is contractible, set 
\[a:= \ \min_{1\leq j \leq l} \ S_k(\gamma_j^{n_*-1}) \, .\]

\noindent By \eqref{uniformity} we can find a natural number $n\geq n_0$ such that $c_n(k)<a$. In any case, Lemma \ref{lemma7.1} implies that 
the value $c_n(k)$ is a critical value of $S_k$ restricted to $\Lambda_{\nu^n_*}$. Thanks to our finiteness assumption 
\[\operatorname{crit}\, S_k\big|_{\Lambda_{\nu^n_*}} \cap \big \{S_k = c_n(k)\big \} \ = \ \T\cdot \gamma_{j_1}^{m_1} \cup \ldots \cup \T\cdot \gamma_{j_r}^{m_r}\]

\noindent for some non-empty set $\{j_1,...,j_r\}\subseteq \{1,...,l\}$ and some positive integers $m_1,...,m_r$. By our choice of $n$, all the $m_i$'s are at least $n_*$. Now we apply Lemma \ref{lemma7.1} with 
\[\mathcal V := \ \mathcal W_{j_1,m_1} \cup \ldots \cup \mathcal W_{j_r,m_r}\]

\noindent and we obtain a path $v\in \mathcal P_n(k)$ with image contained in 
\[\big \{S_k<c_n(k)\big \} \cup \mathcal V\]

\noindent and such that $S_k(v(0))<c_n(k)$. Since also $S_k(v(1))=S_k(\mu^n)<c_n(k)$, and since the sets $\mathcal W_{j_i,m_i}$ have pairwise disjoint closures, the path $v$ is the concatenation of finitely many paths, each of which has 
end-points in $\{S_k<c_n(k)\}$ and is contained in 
\[\big \{S_k<c_n(k)\big \} \cup \mathcal W_{j_i,m_i}\, ,\]
for some $i\in \{1,...,r\}$. By the property of the sets $\mathcal W_{j_i,m_i}$ stated above, the end-points of each of these paths can be joined by paths $w$ in $\{S_k<c_n(k)\}$ and hence by concatenating the $w$'s we get a path in $\{S_k<c_n(k)\}$ which joins $u(0)$ to $u(1)$. Since such a path belongs to $\mathcal P_n(k)$, this contradicts the definition of $c_n(k)$.
\end{proof}

\bibliographystyle{amsalpha}
\bibliography{imbiblio_2}

\end{document}